\documentclass{amsart}
\usepackage{graphicx}
\usepackage{amscd}
\usepackage{color}
\newtheorem{theorem}{Theorem}[section]
\newtheorem{lemma}[theorem]{Lemma}
\newtheorem{claim}[theorem]{Claim}

\newtheorem{conjecture}[theorem]{Conjecture}

\newtheorem{question}[theorem]{Question}

\theoremstyle{definition}

\begin{document}

\title[Characterization of 3-punctured spheres]{Characterization of 3-punctured spheres in non-hyperbolic link exteriors}

\author{Mario Eudave-Mu\~{n}oz}
\address{Instituto de Matem\'{a}ticas, Universidad Nacional Aut\'onoma de M\'exico, Circuito Exterior, Ciudad Universitaria, 04510 M\'exico D.F.,  MEXICO}
\email{mario@matem.unam.mx}
\thanks{The first author is partially supported by grant PAPIIT-UNAM IN101317}

\author{Makoto Ozawa}
\address{Department of Natural Sciences, Faculty of Arts and Sciences, Komazawa University, 1-23-1 Komazawa, Setagaya-ku, Tokyo, 154-8525, Japan}
\email{w3c@komazawa-u.ac.jp}
\thanks{The second author is partially supported by Grant-in-Aid for Scientific Research (C) (No. 26400097, 16H03928, 17K05262), The Ministry of Education, Culture, Sports, Science and Technology, Japan}



\begin{abstract}
In this paper, we characterize non-hyperbolic 3-component links in the 3-sphere whose exteriors contain essential 3-punctured spheres with non-integral boundary slopes.
We also show the existence of embeddings of some multibranched surfaces in the 3-sphere which satisfy some homological conditions to be embedded in the 3-sphere.
\end{abstract}

\maketitle


\section{Introduction}

\subsection{Characterization of 3-punctured spheres with non-integral boundary slopes}

The cabling conjecture (\cite{GS1986}) is a deepest difficult problem in knot theory, which states that if a Dehn surgery on a knot in the 3-sphere yields a reducible 3-manifold, then it is a cable knot and the surgery slope is given by the boundary slope of the cabling annulus.
This conjecture follows the next stronger conjecture; there does not exist an essential $n$-punctured sphere $(n\ge 3)$ with non-meridional boundary slope in any knot exterior in the 3-sphere.

In this paper, we consider an extended problem related to the above stronger conjecture, that is, for a given link exterior in the 3-sphere, does there exist an essential $n$-punctured sphere with non-meridional boundary slope?
However, unlike knot exteriors, there are plenty of essential $n$-punctured spheres with integral boundary slope.
Accordingly, we concentrate essential on $n$-punctured spheres with non-meridional, non-integral boundary slopes.

For hyperbolic links, we propose the following conjecture.

\begin{conjecture}\label{extended}
There does not exist an essential $n$-punctured sphere with non-meridional, non-integral boundary slope in a hyperbolic link exterior in the 3-sphere.
\end{conjecture}

By \cite{GL1987}, any knot exterior in the 3-sphere does not contain an essential $n$-punctured sphere with non-meridional, non-integral boundary slope.
Thus, Conjecture \ref{extended} holds for knots.


Yoshida classfied totally geodesic 3-punctured spheres in an orientable hyperbolic 3-manifold with finitely many cusps (\cite{Y2017}).
This result supports Conjecture \ref{extended}, but we remark that it does not give a complete solution since two hyperbolic link in the 3-sphere might have a same exterior.

A non-hyperbolic link will have an essential annulus or an essential torus in its exterior. For a link $L$ we denote by $N(L)$ a regular neighborhood of $L$, and by $E(L)=S^3-int N(L)$, the exterior of $L$. We have the following characterizations.

\begin{theorem}[Annular case]\label{characterization}
Let $L=C_1\cup C_2\cup C_3$ be a 3-component  link in the 3-sphere $S^3$ whose exterior contains an essential annulus,
and let $P$ be a 3-punctured sphere in the exterior of $L$ such that $l_i\subset \partial N(C_i)$ has a non-meridional, non-integral slope, where $\partial P=l_1\cup l_2\cup l_3$.
Then one of the following cases holds.
\begin{description}
\item[Case 1] $C_1$ is a torus knot contained in a torus $T$, where $T$ decomposes $S^3$ into two solid tori $V_2$ and $V_3$. $C_2$ and $C_3$ are cores of $V_2$ and $V_3$ respectively. 
Let $A$ be an annulus connecting $\partial N(C_2)$ and $\partial N(C_3)$ such that $A$ intersects $T$ in a loop which intersects $C_1$ in one point.
Let $A'$ be an annulus connecting $\partial N(C_2)$ and $\partial N(C_3)$ such that $A$ intersects $T$ in a loop $C_1$.
For $i=2,3$, put $A_i=A\cap V_i\cap E(C_1)$ and $A_i'=A'\cap V_i\cap E(C_1)$, and let $P_i$ be an annulus obtained from $A_i$ by twisting along $A_i'$ for $i=2,3$.
Then $P=P_2\cup P_3$.

\begin{figure}[htbp]
	\begin{center}
	\begin{tabular}{cc}
	\includegraphics[trim=0mm 0mm 0mm 0mm, width=.45\linewidth]{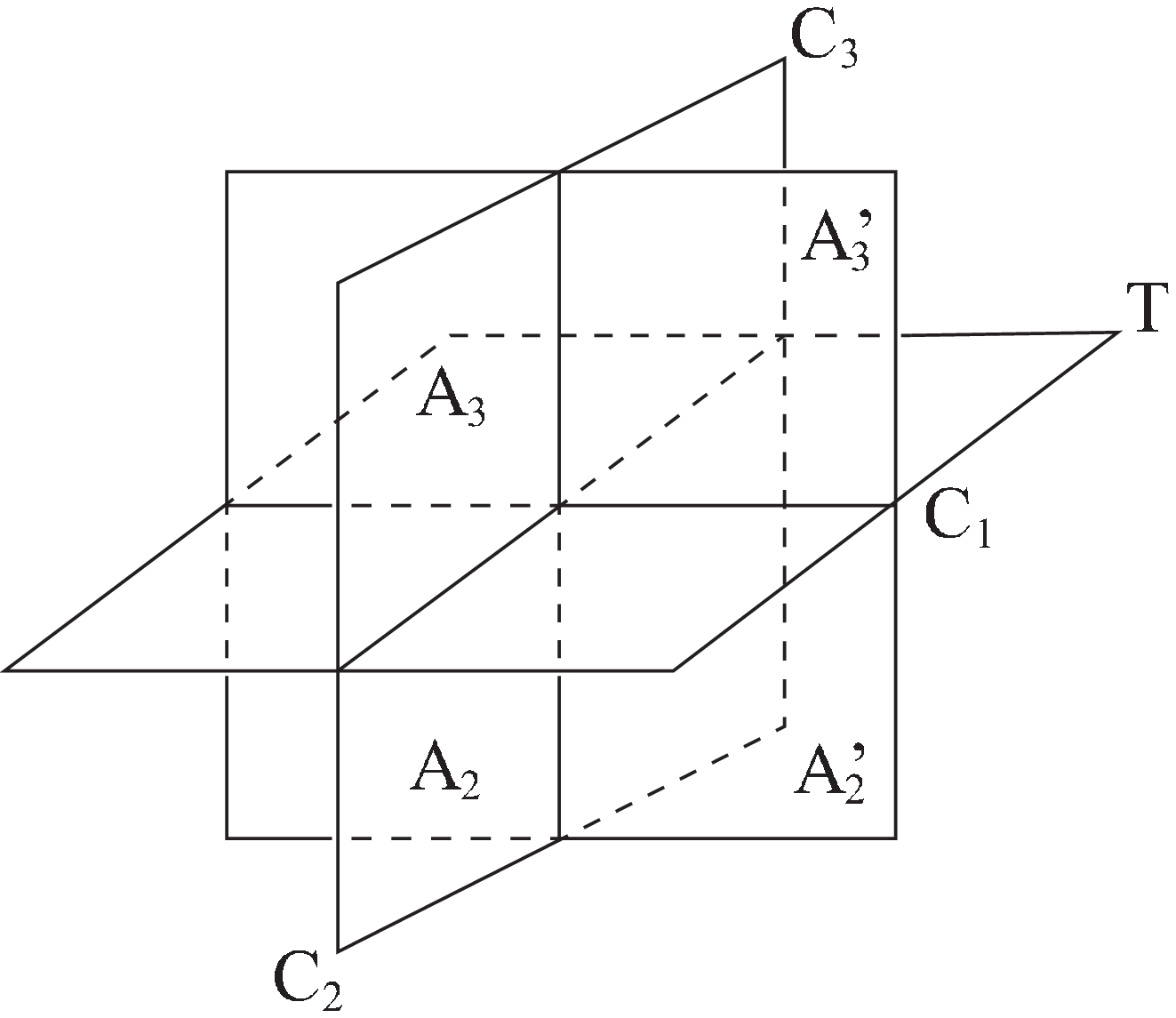}&
	\includegraphics[trim=0mm 0mm 0mm 0mm, width=.45\linewidth]{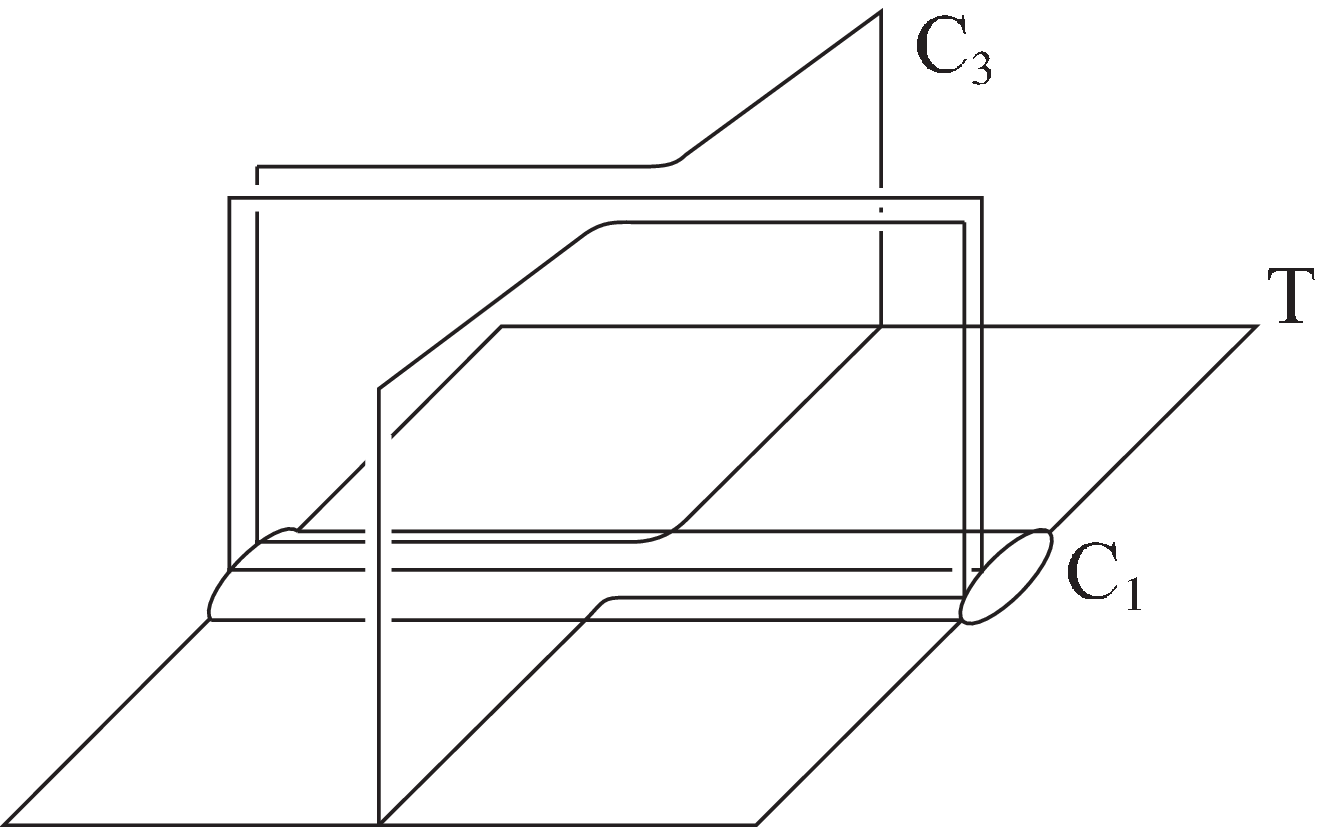}\\
	Configuration of $T,A_i,A_i', C_i$ & $P_3=A_3+2A_3'$
	\end{tabular}
	\end{center}
	\caption{Case 1}
	\label{case1}
\end{figure}

\item[Case 2] $C_2$ is any knot and $C_3$ is any cable knot of $C_2$.
Put $T=\partial N(C_3)$, where $T$ decomposes $S^3$ into a solid torus $V_3$ and the exterior of $C_3$, say $V_2$.
Let $A$ be an annulus connecting $\partial N(C_2)$ and $\partial N(C_3)$ such that $A$ intersects $T$ in a loop, say $\lambda$.
Let $A'$ be an annulus connecting $C_1$ and $\partial N(C_3)$ such that $A'$ intersects $A$ in an arc.
$C_1$ is a cable knot of $C_3$ which is contained in $T$ and intersects $\lambda$ in one point.
Put $A_i=A\cap V_i\cap E(C_1)$ for $i=2,3$, $A_3'=A'\cap V_3\cap E(C_1)$, $P_2=A_2$, and let $P_3$ be an annulus obtained from $A_3$ by twisting along $A_3'$.
Then $P=P_2\cup P_3$.

\begin{figure}[htbp]
	\begin{center}
	\begin{tabular}{cc}
	\includegraphics[trim=0mm 0mm 0mm 0mm, width=.35\linewidth]{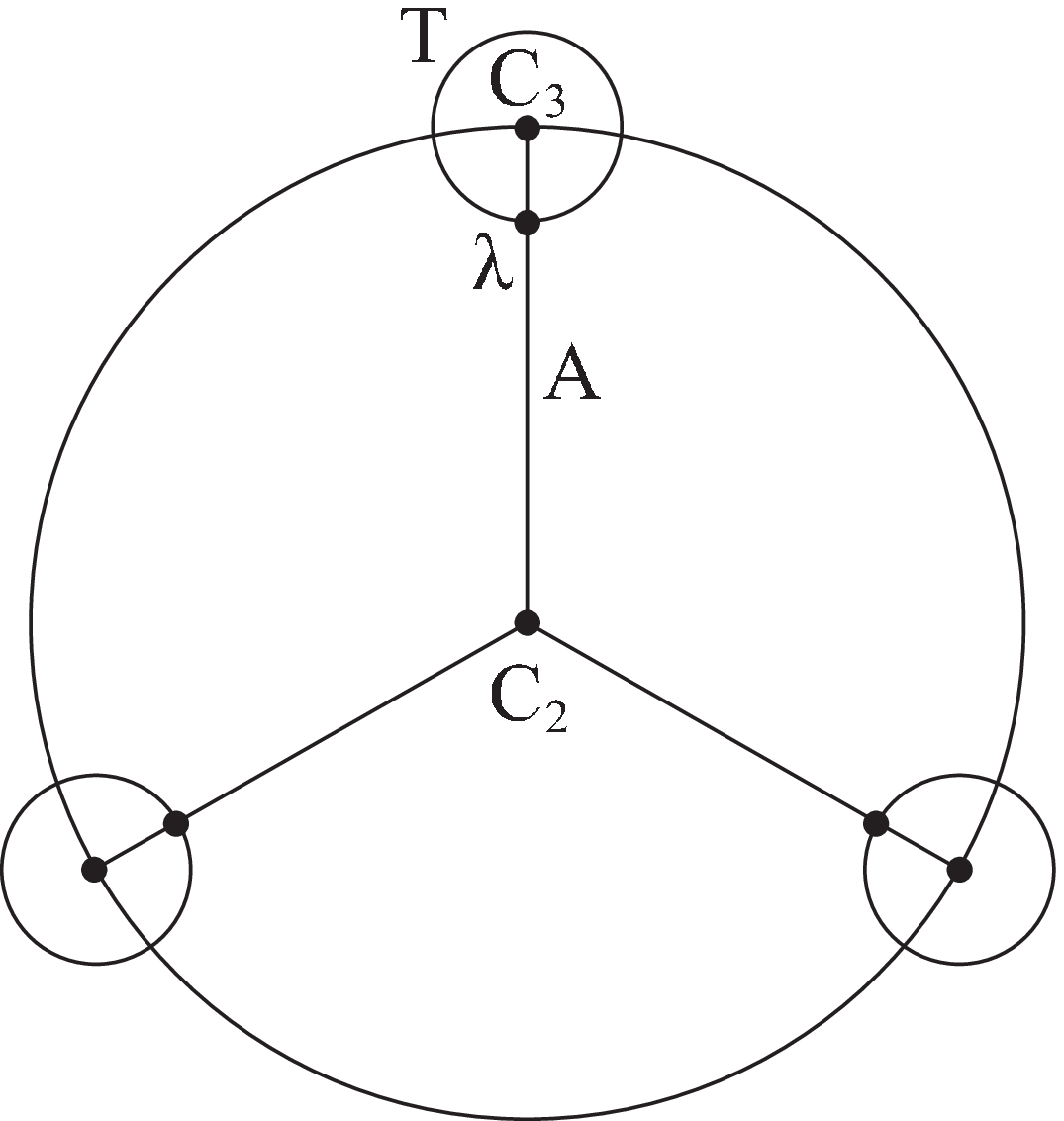}&
	\includegraphics[trim=0mm 0mm 0mm 0mm, width=.45\linewidth]{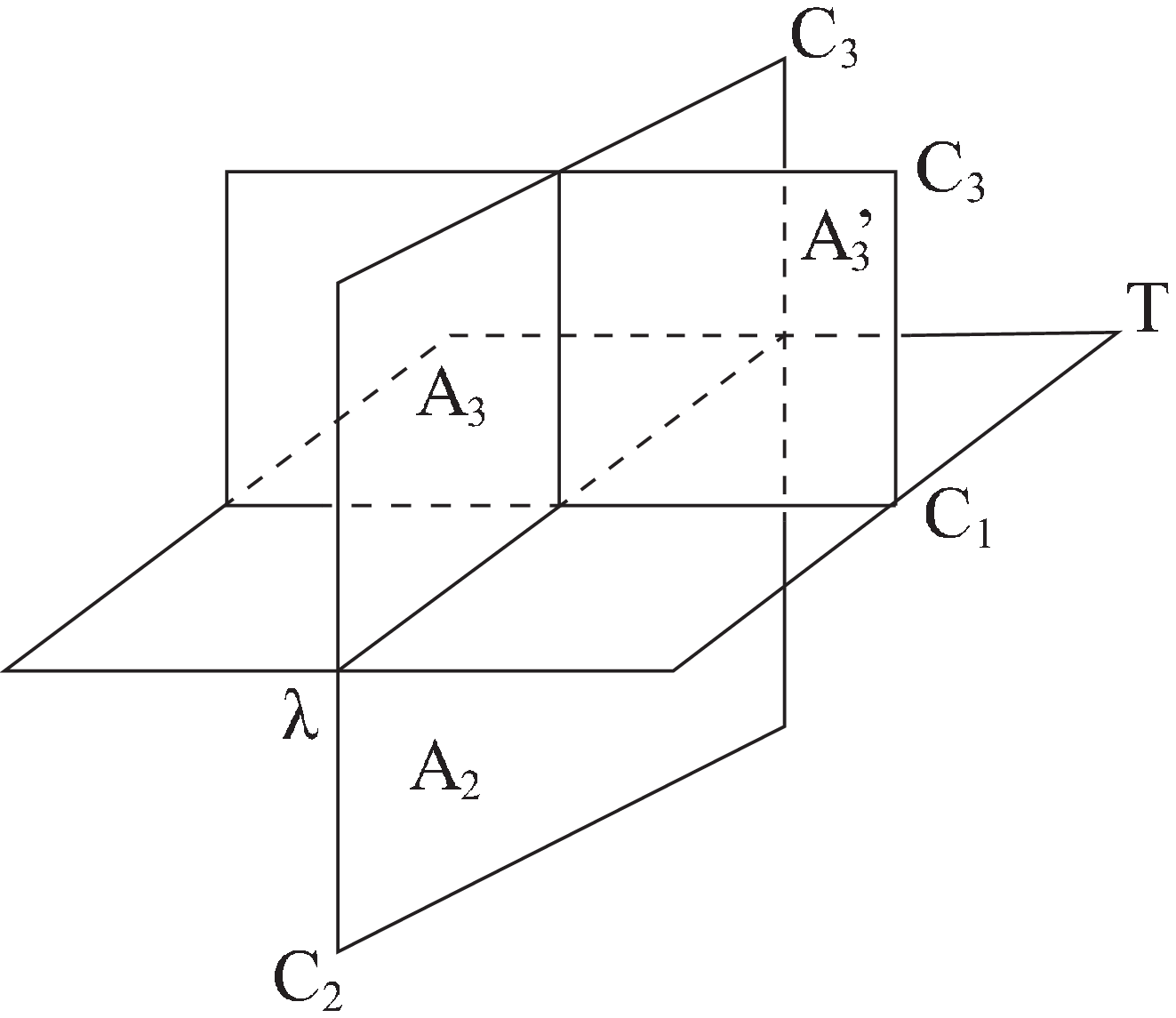}\\
	Configuration of $C_2,C_3,T,A,\lambda$ & Configuration of $T,A_i,A_i', C_i$
	\end{tabular}
	\end{center}
	\caption{Case 2}
	\label{case2}
\end{figure}

\end{description}
\end{theorem}

\begin{theorem}[Toroidal case]\label{characterizationtorus}
Let $L=C_1\cup C_2\cup C_3$ be a 3-component  link in the 3-sphere whose exterior contains an essential torus,
and let $P$ be a 3-punctured sphere in the exterior of $L$ such that $l_i\subset \partial N(C_i)$ has a non-meridional, non-integral slope, where $\partial P=l_1\cup l_2\cup l_3$.
Then one of the following cases holds.

\begin{description}
\item[Case 1] The link $L$ also has an essential annulus and then it has a description as in Theorem \ref{characterization}.

\item[Case 2] $L$ and $P$ can be obtained from one of Case 1 or 2  of Theorem \ref{characterization} by a satellite construction, that is, for a trivial loop $l$ disjoint from $L\cup P$, by re-embedding the exterior $E(l)$ of $l$ into $S^3$ so that it is knotted, $L$ and $P$ are obtained.

\item[Case 3] $L$ and $P$ can be obtained from a satellite construction of a possible hyperbolic link $L'$ admitting a 3-punctured sphere $P'$.
\end{description}
\end{theorem}

\subsection{Embeddings of multibranched surfaces in the 3-sphere}

A {\em multibranched surface} is a 2-dimensional CW complex which is locally homeomorphic to a model as shown in Figure \ref{multibranch}.
We can regard a multibranched surface as a quotient space obtained from a compact 2-manifold ({\em sectors}) by gluing its boundary to a closed 1-manifold ({\em branches}) via covering maps.
We say that a multibranched surface is {\em regular} if for each branch, covering degrees of all covering maps are uniform. A systematic study of some types of multibranched surfaces is done in \cite{GGH2016} and \cite{GGH2018}, where the surfaces are called stratifolds. Here we are interested in the embedability of multibranched surfaces in the 3-sphere.

\begin{figure}[htbp]
	\begin{center}
	\includegraphics[trim=0mm 0mm 0mm 0mm, width=.4\linewidth]{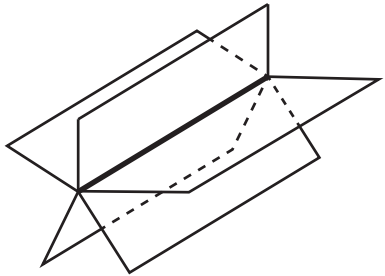}
	\end{center}
	\caption{A local model of multibranched surfaces}
	\label{multibranch}
\end{figure}

The Menger--N\"{o}beling theorem ({\cite[Theorem 1.11.4.]{E1978}}) shows that any finite 2-dimensional CW complex can be embedded in the 5-dimensional Euclidian space $\Bbb{R}^5$.
This is a best possible result since for example, the union of all 2-faces of a 6-simplex cannot be embedded in $\Bbb{R}^4$ ({\cite[1.11.F]{E1978}}).
In contrast to this result, it was shown that any multibranched surface can be embedded in $\Bbb{R}^4$ (\cite{G1987}, {\cite[Proposition 2.3]{MO2017}}).
Furthermore, a multibranched surface is embeddable in some closed orientable 3-dimensional manifold if and only if the multibranched surface is regular ({\cite[Proposition 2.7]{MO2017}}).
We remark that any 3-manifold can be embedded in $\Bbb{R}^5$ (\cite{W1956}).
Thus we obtain the next diagram on the embedability of multibranched surfaces (Figure \ref{embedding}).

\begin{figure}[htbp]
	\begin{center}
	\includegraphics[trim=0mm 0mm 0mm 0mm, width=.8\linewidth]{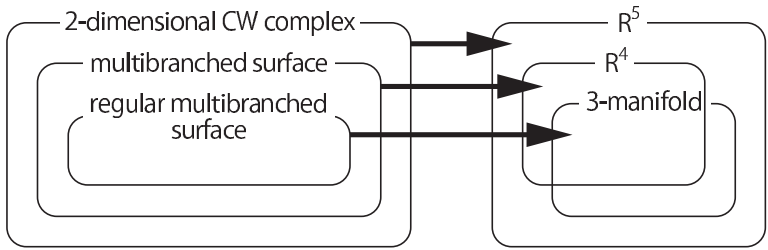}
	\end{center}
	\caption{The embeddability of multibranched surfaces}
	\label{embedding}
\end{figure}

It is convenient to represent a multibranched surface with only orientable sectors as a weighted bipartite graph.
We regard the set of sectors as the set of weighted vertices, where each vertex has a weight by the genus of the sector.
Another vertex set represents the set of branches.
If a sector is glued with a branch by a covering map, then the corresponding vertex is joined to the corresponding vertex by an edge with the weight by the covering degree.
For example, a weighted bipartite graph as shown in Figure \ref{representation} represents a multibranched surface with one sector which is a compact orientable surface of genus $g$ with $n$ boundary components and $n$ branches, and the covering degrees are $p_1,\ p_2,\ \ldots,\ p_n$.

Theorem \ref{characterization} gives a partial solution of \cite[Problem, p.631]{MO2017}, that is, for a multibranched surface $X=X_g(p_1,p_2,\ldots,p_n)$ which has a graph representation as shown in Figure \ref{representation}, the embeddability into the 3-sphere is determined by the torsion-freeness of the first homology group.
Since $H_1(X)=\left( \mathbb{Z}/p \mathbb{Z} \right) \oplus \mathbb{Z}^{2g+n-1}$ as shown in \cite[Example 4.3]{MO2017}, where $p={\rm gcd}\{ p_1, \cdots, p_n\}$, if $p>1$, then $X$ cannot be embedded into the 3-sphere.

\begin{figure}[htbp]
	\begin{center}
	\includegraphics[trim=0mm 0mm 0mm 0mm, width=.35\linewidth]{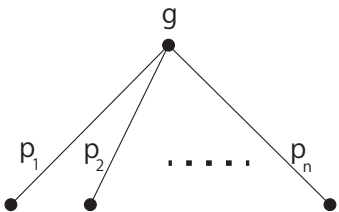}
	\end{center}
	\caption{A graph representation of a multibranched surface $X_g(p_1,p_2,\ldots,p_n)$}
	\label{representation}
\end{figure}

Conversely, if $p=1$, then can $X$ be embedded into the 3-sphere?
As it was announced in the footnote in \cite[Problem, p.631]{MO2017}, we show the following.

\begin{theorem}\label{solution}
If $p=1$, then $X_g(p_1,p_2,p_3)$ can be embedded into the 3-sphere for a sufficiently large $g$.
\end{theorem}

Theorem \ref{solution} can be proved by constructing a suitable Seifert surface cobounded by a 3-component cable link.

Next, we consider whether the genus $g$ can be always taken to be $0$.

Let $L=C_1\cup C_2 \cup C_3$, $A$, $A'$, $A_2$, $A_3$, $A_2'$, $A_3'$, $P$, $P_2$, $P_3$ as in Case 1 of Theorem \ref{characterization}. 
Assume that $\partial A$ has slope $p\lambda_2 +q\mu_2$ on $\partial N(C_2)$ and slope $-q\lambda_3 -p\mu_3$ on $\partial N(C_3)$.
Assume that $\partial A'$ has slope $r\lambda_2 +s\mu_2$ on $\partial N(C_2)$ and slope $-s\lambda_3 -r\mu_3$ on $\partial N(C_3)$.
Then we must have $ps-qr=\pm 1$. Now assume that $P_2$ is obtained by $n_2$ twists of $A_2$ around $A_2'$, and that 
$P_3$ is obtained by $n_3$ twists of $A_3$ around $A_3'$. Remember that $P=P_2\cup P_3$. It follows that $\partial P$ has slope 
$(p+n_2r)\lambda_2 + (q+n_2s)\mu_2$ on $\partial N(C_2)$, slope $(-q-n_3s)\lambda_3 + (-p-n_3r)\mu_3$ on $\partial N(C_3)$,
and slope $(n_3-n_2)\lambda_1 - \mu_1$ on $\partial N(C_1)$, where $\lambda_1$ is given by the embedding of $C_1$ on the torus $T$.

So, if $X_0(p_1,p_2,p_3)$ is embedded in the 3-sphere as in Case 1 of Theorem \ref{characterization}, then there exists $p,q,r,s,n_2,n_3$
such that $p_1=n_3-n_2$, $p_2=p+n_2r$ and $p_3=-q-n_3s$, where $ps-qr=\pm 1$. From this follows that $p_1,p_2,p_3$ must satisfy
$sp_2+rp_3+srp_1=\pm 1$. It can be shown that if $X_0(p_1,p_2,p_3)$ is embedded in the 3-sphere as in Case 2 of Theorem \ref{characterization},
then there is an integer $t$ such that $p_3=1+tp_1$, and $p_2$ is arbitrary, but subject to $1={\rm gcd}\{ p_1, p_2, p_3\}$.

It is not too difficult to show that the triple $\{5,7,18\}$ does not satisfy the above conditions for $\{ p_1, p_2, p_3\}$. So, it could be that
$X_0(5,7,18)$ cannot be embedded in the 3-sphere. So, we ask:
 

\begin{question}\label{genus0}
If $p=1$, can $X_0(p_1,p_2,p_3)$ can be embedded into the 3-sphere? If not, which is the minimal $g$ for which $X_g(p_1,p_2,p_3)$ can be
embedded in the 3-sphere?
\end{question}




\section{Proof of Theorems \ref{characterization} and \ref{characterizationtorus}}


\subsection{Annular case} 

Now we proceed to give a proof of Theorem \ref{characterization}. Let $L=C_1\cup C_2\cup C_3$ be a 3-component non-hyperbolic link in the 3-sphere $S^3$,
and $P$ be a 3-punctured sphere in the exterior $E(L)$ of $L$ with $\partial P=l_1\cup l_2\cup l_3$ such that $l_i$ has a slope $q_i/p_i$, where $p_i>1$ $(i=1,2,3)$. Note first that $P$ is incompressible, for by doing a compression we would get a disk whose boundary have non-integral slope on some $C_I$. This also implies that the link $L$ is non-splitable.

Suppose that $E(L)$ contains an essential annulus $A$ with $\partial A=a_1\cup a_2$. 
There are seven cases.

\begin{description}
\item[Case A] Both $a_1$ and $a_2$ have integral slopes on $\partial N(C_1)$.
\item[Case B] Both $a_1$ and $a_2$ have meridional slope on $\partial N(C_1)$.
\item[Case C] $a_1$ has meridional slope on $\partial N(C_1)$ and $a_2$ has an integral slope on $\partial N(C_2)$, .
\item[Case D] $a_1$ has an integral slope on $\partial N(C_1)$, and $a_2$ has a non-integral slope on $\partial N(C_2)$.
\item[Case E] $a_1$ has an integral slope on $\partial N(C_1)$, and $a_2$ has an integral slope on $\partial N(C_2)$.
\item[Case F] Both $a_1$ and $a_2$ have non-integral slopes on $\partial N(C_1)$ and $\partial N(C_2)$ respectively.
\item[Case G] Both $a_1$ and $a_2$ have non-integral slopes on $\partial N(C_1)$.
\end{description}

\vspace{1em}

In Case A, it follows that either $A$ is a cabling annulus for $C_1$, that is, $A$ is an essential annulus in $E(C_1)$, and then $C_1$ is a non-trivial tours knot or cable knot, or well, $A$ is a trivial annulus in $E(C_1)$. Let $T$ be a torus consisting of the union of $A$ and an annulus $A'$ contained in $\partial N(C_1)$; we can consider that the torus $T$ contains $C_1$. Then $C_1$ is a $(p,q)$-cable knot on $T$, with the possibility that $p=1$.

Since $l_1$ has a non-integral slope and $a_i$ has an integral slope on $\partial N(C_1)$, we may assume that $P\cap A$ consists of arcs $\alpha_1,\ldots,\alpha_n$ such that $\alpha_i$ is an essential arc in both $P$ and $A$ $(i=1,\ldots,n)$ which connects two points in $l_1$.
Let $Q_2$ be a region of $P-\alpha_2$ cobounded by $l_2$ and $\alpha_n$ with a subarc in $l_1$, and $Q_3$ be a region of $P-\alpha_1$ cobounded by $l_3$ and $\alpha_1$ with a subarc in $l_1$.

\begin{figure}[htbp]
	\begin{center}
	\includegraphics[trim=0mm 0mm 0mm 0mm, width=.3\linewidth]{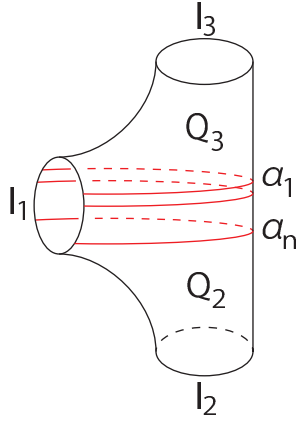}
	\end{center}
	\caption{Arcs $\alpha_1,\ldots,\alpha_n$ in $P$ and regions $Q_2,\ Q_3$}
	\label{P}
\end{figure}

First, assume that $n\ge 2$, and let $R_i$ be the region of $P-(\alpha_1\cup\cdots\cup\alpha_n)$ between $\alpha_i$ and $\alpha_{i+1}$ $(i=1,\ldots,i-1)$.

\begin{lemma}\label{cable}
A region $R_i$ implies that $C_1$ is a $(2,q)$-cable knot on $T$.
\end{lemma}

\begin{proof}
Since $\partial R_i$ consists of $\alpha_i$, $\alpha_{i+1}$, and two arcs in $\partial N(C_1)$, $R_i$ extends a compressing disk for $T$ which intersects $C_1$ in two points.
See Figure \ref{R} for the configulation.
Hence $C_1$ is $(2,q)$-cable knot on $T$.

\begin{figure}[htbp]
	\begin{center}
	\includegraphics[trim=0mm 0mm 0mm 0mm, width=.4\linewidth]{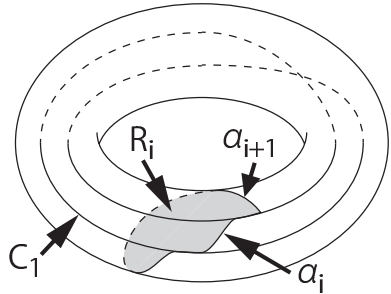}
	\end{center}
	\caption{A region $R_i$ implies that $C_1$ is a $(2,q)$-cable knot on $T$.}
	\label{R}
\end{figure}
\end{proof}

\begin{claim}
$n=1$.
\end{claim}

\begin{proof}
If $n>2$, then there are two compressing disks for $T$ in both sides which come from $R_i$ and $R_{i+1}$.
By Lemma \ref{cable}, this shows that $C_1$ is a $(2,2)$-cable link on $T$, a contradiction.

Next, we consider the case $n=2$.
By Lemma \ref{cable}, the region $R_1$ extends a compressing disk for $T$.
Put $S^3=V_2\cup_T V_3$, where $R_1 \subset V_2$, and then $V_2$ must be a solid torus.
Then $C_2$ and $C_3$ are contained in $V_3$.
The region $Q_2$ is an annulus cobounded by $l_2$ and $\alpha_2$ with a subarc in $l_1$, and $l_2$ wraps around $C_2$. Note that the union of $\alpha_2$ and the subarc of $l_1$ wraps around $T$ at least twice longitunally, for otherwise $l_1$ would be a longitudinal slope on $C_1$. This implies that $V_3$ is a solid torus and that $C_2$ is its core. Similarly,  $l_3$ is a core of the solid torus $V_3$. This results in a contradiction.

Therefore, we have $n=1$.
\end{proof}

In the following, we assume that $n=1$. Put $S^3=V_2\cup_T V_3$, and we assume that $V_2$ is a solid torus. Assume that $Q_2 \subset V_2$ and that
$Q_3 \subset V_3$. 

Let $\alpha_2'$ be the loop $\partial Q_2-l_2$, that is, $\alpha_2'$ is the union of $\alpha_1$ and an arc of $l_1$. The arc $\alpha_2'$ is parallel to $l_2$, and then $C_2$ is a core of $V_2$. The arc $\alpha_2'$ wraps $l_2$ times longitudinally around $T$. Let $\gamma_2$ be an arc in $\partial N(C_1)\cap V_2$ which connects two points of $\partial \alpha_1$, and such that $\alpha_1 \cup \gamma_2$ wraps longitudinally around $T$ a minimal number of times among all such arcs in $\partial N(C_1)\cap V_2$. 
By an isotopy of $P$, we assume that $\alpha_2'\cap N(C_1)$ intersects $\gamma_2$ minimally.
Then it can be observed that $\alpha_2'$ is obtained from $\alpha_1\cup \gamma_2$ by Dehn twists along the core of $\partial N(C_1)\cap V_2$. Let $A_2$ be an annulus in 
$V_2-int N(C_2)$, where $\partial A_2 =(\alpha_1\cup \gamma_2) \cup \gamma_2'$, where $\gamma_2'$ is a curve on $\partial N(C_2)$, and let $A_2'$ be another annulus in
$V_2-int N(C_2)$, where one boundary component is the core of $\partial N(C_1)\cap V_2$, and the other is a curve on $\partial N(C_2)$. Note that $Q_2$ and $A_2'$ intersect in one arc, and that $Q_2$ is obtained by twisting $A_2$ along $A_2'$.

If $V_3$ is also a solid torus, then $C_3$ is a core of $V_3$, and by a similar construction we have annuli $A_3$, $A_3'$, and an arc $\gamma_3 \subset \partial N(C_1)\cap V_3$,
so that $\gamma_2 \cup \gamma_3$ is a meridian of $N(C_1)$, then the union of $A_2$, $A_3$ and a meridian of $N(C_1)$ is a spanning annulus between $C_2$ and $C_3$. It follows also that $Q_3$ is obtained by twisting $A_3$ along $A_3'$. Then we have the conclusion of Case 1 in Theorem \ref{characterization}.

Suppose now that $V_3$ is not a solid torus, that is, $V_2$ is knotted. Let $\alpha_3'$ be the loop $\partial Q_3-l_2$, that is, $\alpha_3'$ is the union of $\alpha_1$ and an arc of $l_1$. Note that $\alpha_3'$ must be a longitude of $V_2$, and that $C_2$ is a cable around $C_3$. It follows that we have the conclusion of Case 2 in Theorem \ref{characterization}.





\vspace{1em}
Suppose now we have Case B. In this case the annulus $A$ extends to a decomposing sphere for $L$. As $l_1$ has a non-integral slope and $a_i$ has meridional slope on $\partial N(C_1)$, we may assume that $P\cap A$ consists of arcs $\alpha_1,\ldots,\alpha_n$, such that $\alpha_i$ is an essential arc in both $P$ and $A$ $(i=1,\ldots,n)$ which connects two points in $l_1$. Note that $l_1$ intersects each $a_i$ in at least two points, which implies that $n\geq 2$. As in the previous case, there is a region $R$ between the arcs $\alpha_1$ and $\alpha_2$. The curve $\partial R$ defines a non-integral slope for the torus formed by the union of $A$ and a meridional annulus bounded by $a_1\cup a_2$, which is not possible.

\vspace{1em}
In case C, it follows that there is also an annulus $A$ as in Case B.

\vspace{1em}
In Case D, it follows that $C_1$ is a cable of $C_2$ and there exists a cabling annulus for $C_1$ disjoint from $A$.
This arrives at Case A.

\vspace{1em}
In Case E, it follows that there exists an essential annulus $A'$ in $E(L)$ with $\partial A'=a_1'\cup a_2'$ such that $a_i'$ has an integral slope on $\partial N(C_1)$ $(i=1,2)$.
This also arrives at Case A.

\vspace{1em}
In Case F, $C_1\cup C_2$ is a Hopf link.
Put $V=E(C_1\cup C_2)-int N(A)$, which is a solid torus with a non-trivial torus knot core. Note that $C_3 \subset V$.
Suppose first that $P$ and $A$ have non-empty intersection. Then $P\cap A$ consists of $n_{12}$ essential arcs connecting $l_1$ and 
$l_2$, and $n_3$ essential loops parallel to $l_3$.

\begin{claim}
$n_{12}=1$ and $n_3=0$.
\end{claim}

\begin{proof}

Suppose that $n_3\not= 0$, let $\alpha$ be an outermost loop of $P\cap A$.
Since $\alpha$ is an essential loop of $A$, $\alpha$ is parallel to the core of $V$.
On the other hand, since $\alpha$ is parallel to $l_3$, $\alpha$ is a cable knot of $C_3$,
but this is not possible. Therefore $n_3=0$.

If $n_{12}\ge 2$, then by Lemma \ref{cable}, a rectangle region $R$ of $P-A$ between two essential arcs gives a meridian disk of $V$.
Since $R\cap C_3=\emptyset$, $C_3$ is contained a 3-ball disjoint from $C_1$ and $C_2$, that is, $L$ is splitable, which is not possible.
\end{proof}

If $n_{12}=1$, then the loop of $P\cap \partial V$ is parallel to $l_3$ and hence it is a cable of $C_3$.
This implies that $C_3$ is a core of $V$.
Therefore, $C_3$ bounds a cabling annulus disjoint from $C_1\cup C_2$, and we can proceed to Case A.

Suppose now that $P\cap A= \emptyset$. Then $P$ lies in the solid torus $V$. Both, $l_1$ and $l_2$ must be longitudes of $V$.
Re-embed $V$ in $S^3$ so that $l_1$ and $l_2$ are preferred longitudes of $V$. By taking disks bounded by $l_1$ and $l_2$, $P$
becomes a disk with a non-integral slope in $C_3$, which is no possible.

\vspace{1em}

Finally consider Case G. In this case the annulus $A$ is parallel to an annulus $A'$ contained in $\partial N(C_1)$, and then $A\cup A'$
bound a solid torus $V$, which must contain one or both of $C_2$, $C_3$. Consider again the intersections between $A$ and $P$.
If $A\cap P \not= \emptyset$, by an argument similar to Case A, we have that that the intersection consist of just an arc, and then $P$ is
divided in two regions $Q_2$ and $Q_3$. One of the regions, say $Q_2$, is contained in $V$, and one component of $\partial Q$ must be
a curve which is a cable of $V$, and then $C_2$ will be a core of $V$. But then, there is an annulus with one boundary component
on $C_1$ with non-integral slope, and one boundary component on $C_2$ with integral slope. We refer then to Case D.

If $P\cap A=\emptyset$, then $P$ is contained in $V$. It follows that $l_1$ is a core of $V$. Re-embed $V$ is $S_3$ such that $A'$ is
a meridional annulus of a trivial knot $C_1'$. Then $P$ union a meridian of $C_1'$ becomes an annulus $A''$. As the boundary of $A''$
consist of curves on non-integral slope on the new $C_2'$ and $C_3'$, it follows that $C_2'\cup C_3'$ form a Hopf link. Now, $C_1'$
is a trivial knot which intersects $A''$ in one point, and then it can be isotoped to lie in a level torus, that is, a torus which intersects
$A''$ in one close curve and is isotopic to $\partial N(C_2')$. Then there must be another annulus $A$, disjoint from $C_1'$ whose boundary
consist of curves on $C_2$ and $C_3$, where at least one of these must have integral slope. Then $C_2$ and $C_3$ bound an annulus disjoint
from $C_1$, with at least one boundary component with a integral slope. We refer then to Case $D$ or $E$. 

Finally note that Cases B and C are impossible, but that the remaining cases are indeed possible. Case A, D, E, G appear in Cases 1 and 2 of
Theorem \ref{characterization}, while Case F appears in Case 1.

\subsection{Toroidal case}
Next, we suppose that $E(L)$ contains an essential torus $T$.

We may assume that $P\cap T$ consists of essential loops in both $P$ and $T$.
Let $A$ be an outermost annulus in $P$ with respect to $P\cap T$, and let $T'$ be an annulus in $E(L)$ obtained from $T$ by an annulus compression along $A$.
Since $p_i>1$, the boundary slope of $T'$ is non-integral.
If $T'$ is boundary parallel in $E(L)$, this implies that $T$ is also boundary parallel in $E(L)$, a contradiction. Then $T'$ must be an essential annulus in $E(L)$,
and then $L$ admits a description as in Theorem \ref{characterization}.

Suppose now $P\cap T=\emptyset$.
Put $S^3=V_1\cup_T V_2$, where $V_1\supset L\cup P$.
Then $V_1$ is a solid torus and $V_2$ is a non-trivial knot exterior.
By exchanging $T$ if necessary, we may assume that $(V_1-int N(L))-P$ does not contain an essential torus. 
We re-embed $V_1$ in $S^3$ so that it is unknotted, that is, $V_2$ is a trivial knot exterior.
Then we obtain an atoroidal 3-component link $L'$ and a 3-punctured sphere $P'$ in $E(L)$ with non-meridional, non-integral slope.
There are two case, either there is an essential annulus in $E(L')$  and then $L'$ and $P'$ satisfy Case 1 or Case 2  of Theorem \ref{characterization}, or $L'$ is a hyperbolic link.
In any case,  $L$ and $P'$ are obtained from $L'$ and $P'$ by a satellite construction.

\section{Proof of Theorem \ref{solution}}

Let $p_1,p_2,p_3$ be positive integers greater than 1 such that $p=\gcd\{p_1,p_2,p_3\}=1$.
We will construct a 3-component link $L=l_1\cup l_2\cup l_3$ in $S^3$ whose exterior contains 3-punctured compact orientable surface $F$ of genus $g$ for a sufficiently large $g$ with $\partial F=f_1\cup f_2\cup f_3$ such that $f_i$ has a slope $q_i/p_i$ on $\partial N(l_i)$.

Let $k_i$ be a $(p_i,q_i)$ cable of $l_i$ for $i=1,2,3$, and put $K=k_1\cup k_2\cup k_3$.
We will construct a Seifert surface for $K$ which is disjoint from the interior of $N(L)$,
and obtain a multibranched surface $X_g(p_1,p_2,p_3)$.

Put $a_{ij}=lk(l_i,l_j)$.
Then we have
\begin{eqnarray*}
lk(l_1,K) &=& q_1 + a_{12}p_2 + a_{13}p_3 \\
lk(l_2,K) &=& q_2 + a_{12}p_1 + a_{23}p_3 \\
lk(l_3,K) &=& q_3 + a_{13}p_1 + a_{23}p_2
\end{eqnarray*}

\begin{lemma}
If $lk(l_1,K)=lk(l_2,K)=lk(l_3,K)=0$, then there exists a Seifert surface $S$ for $K$ which is disjoint from the interior of $N(L)$.
\end{lemma}

\begin{proof}
Let $S$ be a Seifert surface for $K$.
We may assume that $S$ intersects $L$ transversely.
Since $lk(l_i,K)=0$, by successively tubing $S$ along $l_i$ for adjacent two points with opposite orientation, we obtain a Seifert surface for $K$ which is disjoint from $L$.

Let $A_i$ be an annulus in $N(k_i)$ connecting $k_i$ and $l_i$. If $A_1$ is disjoint from $S$, then $S$ can be isotoped to be disjoint from $int N(L)$. If not, then $S$ intersects $A_i$ in curves isotopic to $k_i$. By isotoping $S$ across $k_i$ we get a surface $S'$ disjoint from $int N(L)$ but which may have self-intersections. By  doing an oriented double curve sum along the curves of intersections, we get a Seifert surface as desired.
\end{proof}

To get $lk(l_1,K)=0$, just put $q_1=-a_{12}p_2 - a_{13}p_3$. 
But we need also that $\gcd(p_1,q_1)=1$.
We consider similarly for $l_2,l_3$.
This is solved by the following Lemma.

\begin{lemma}\label{algebraic}
Let $p_1,p_2,p_3$ be positive integers greater than 1 such that $\gcd(p_1,p_2,p_3)=1$.
Then there exists integers $a_{12},a_{13},a_{23}$ such that
\begin{eqnarray*}
\gcd(p_1,a_{12}p_2+a_{13}p_3) &=&1\\
\gcd(p_2,a_{12}p_1+a_{23}p_3) &=&1\\
\gcd(p_3,a_{13}p_1+a_{23}p_2) &=&1
\end{eqnarray*}
\end{lemma}

\begin{proof}
For simplicity of notation, we rewrite $a=p_1,\ b=p_2,\ c=p_3$ and $x=a_{12},\ y=a_{13},\ z=a_{23}$.
Thus Lemma \ref{algebraic} is rewritten as

\begin{lemma}
For given integers $a,b,c$, $(a,b,c)=1$, $a,b,c>1$,
there exist integers $x,y,z$ such that 
\begin{eqnarray*}
(a,N_1)=1, \text{where}\ N_1=bx+cy\\
(b,N_2)=1, \text{where}\ N_2=ax+cz\\
(c,N_3)=1, \text{where}\ N_3=ay+bz
\end{eqnarray*}
\end{lemma}

Without loss of generality, we can assume that $c$ is an odd number.

Let $a=\alpha c_1$ where $(a,c)=1$ and any prime factor of $c_1$ divides $c$, possibly $\alpha=1$ or $c_1=1$.
Let $b=\beta c_2$ where $(\beta,c)=1$ and any prime factor of $c_2$ divides $c$, possibly $\beta=1$ or $c_2=1$.
Let $c=\gamma c_3$ where $(\gamma,a)=1$, $(\gamma,b)=1$ and any prime factor of $c_3$ divides $a$ or $b$. Note that any prime factor of $c_3$ cannot divide both $a$ and $b$ since $(a,b,c)=1$.
Let 
\begin{eqnarray*}
x &=& \alpha \beta\\
y &=& -b\alpha \beta +\gamma\\
z &=& -a \alpha \beta +\gamma
\end{eqnarray*}
Then 
\begin{eqnarray*}
N_1 &=& b\alpha \beta -cb \alpha \beta +c\gamma\\
N_2 &=& a\alpha \beta -ca \alpha \beta +c\gamma\\
N_3 &=& -ab \alpha \beta +a \gamma -ba \alpha \beta +b\gamma\\
 &=& -2ab \alpha \beta +a\gamma +b\gamma
\end{eqnarray*}

\begin{claim}
$(N_1,a)=1$
\end{claim}

\begin{proof}
Let $p$ be a prime factor of $a$.
Remember $a=\alpha c_1$.
\begin{description}
\item[Case 1] $p|\alpha$\\
If $p|N_1$, then $p|N_1-b\alpha\beta+cb\alpha\beta$, that is, $p|c\gamma$, but this is not possible since $(\alpha,c)=1$ and $(\alpha,\gamma)=1$.

\item[Case 2] $p|c_1$\\
If $p|N_1$, then $p|b\alpha\beta$, but $p\not|\alpha$ since $(\alpha,c_1)=1$, and $p\not|\beta$ since $p|a$ and $p|c$.
\end{description}
\end{proof}

Similarly we have

\begin{claim}
$(N_2,b)=1$
\end{claim}

\begin{claim}
$(N_3,c)=1$
\end{claim}

\begin{proof}
Let $p$ be a prime factor of $c$.
Remember $c=\gamma c_3$.
\begin{description}
\item[Case 1] $p|\gamma$\\
If $p|N_3$, then $p|2ab\alpha\beta$, but $p\not| 2$ since $p$ is odd, and $p\not|a$, $p\not|b$ since $(\gamma,a)=1$, $(\gamma,b)=1$.
\item[Case 2] $p|c_3$\\
If $p|N_3$, then $p|a$ or $p|b$, say $p|a$.
It follows that $p|b\gamma$.
But $p\not | b$ since otherwise $a,b,c$ would have a common factor, and $p\not | \gamma$ by the choice of $\gamma$ and $c_3$.
\end{description}
\end{proof}

This completes the proof.
\end{proof}

Finally, given $a_{12},a_{13},a_{23}$, it is easy to find a link $L=l_1 \cup l_2 \cup l_3$
such that $a_{ij}=lk(l_i,l_j)$. For example, a certain closed pure 3-braid will suffice.


\bibliographystyle{amsplain}

\begin{thebibliography}{10}

\bibitem{E1978} R. Engelking, {\em Dimension Theory}, North-Holland, Amsterdam, 1978.

\bibitem{G1987} D. Gillman, {\em Generalising Kuratowski's theorem from $\Bbb{R}^2$ to $\Bbb{R}^4$}, Ars Comb. 23A (1987) 135--140.

\bibitem{GGH2016} J.C. G\'omez-Larra\~aga, F. Gonz\'alez-Acu\~na, W. Heil, 
{\em 2-dimensional stratifolds}, In book:
Castrill\'{o}n L\'{o}pez, Marco (ed.) et al., A mathematical tribute to Professor Jos\'{e} Mar\'{i}a Montesinos Amilibia on the occasion of his seventieth birthday. Madrid: Universidad Complutense de Madrid, Facultad de Ciencias Matem\'{a}ticas, Departamento de Geomet\'{i}a y Topolog\'{i}a, pp. 395--405 (2016).

\bibitem{GGH2018} J.C. G\'omez-Larra\~aga, F. Gonz\'alez-Acu\~na, W. Heil, {\em Classification of simply-connected trivalent 2-dimensional stratifolds}, Topol. Proc. 52, 329--340 (2018). 

\bibitem{GS1986} F. Gonz\'{a}lez-Acu\~{n}a, H. Short, {\em Knot surgery and primeness}, Math. Proc. Cambridge Philos. Soc. {\bf 99} (1986), 89--102.

\bibitem{GL1987} C. McA. Gordon, J. Luecke, {\em Only integral Dehn surgeries can yield reducible
manifolds}, Math. Proc. Camb. Phil. Soc. {\bf 102} (1987) 94--101.

\bibitem{MO2017} S. Matsuzaki, M. Ozawa, {\em Genera and minors of multibranched surfaces}, Topology and its Appl. {\bf 230} (2017) 621--638.

\bibitem{Y2017} K. Yoshida, {\em Union of 3-punctured spheres in hyperbolic 3-manifolds}, arXiv:1708.03452.

\bibitem{W1956} C. T. C. Wall, {\em All 3-manifolds imbed in 5-space}, Bull. Amer. Math. Soc., {\bf 71} (1956) 564--567.

\end{thebibliography}

\end{document}